\documentclass[12pt]{amsart}
\usepackage{graphicx}
\usepackage{amsmath, amscd, amssymb, amsthm}
\usepackage[subnum]{cases}
\usepackage{changepage}
\usepackage{xcolor}
\usepackage{marginnote}
\usepackage{hyperref}
\usepackage{cleveref}
\usepackage[all]{xy}
\usepackage{tabularx}
\usepackage{ltablex}
\usepackage{longtable}
\usepackage[T1]{fontenc}

\newtheorem{theorem}{Theorem}[section]
\newtheorem{lemma}[theorem]{Lemma}

\newtheorem{proposition}[theorem]{Proposition}
\newtheorem{corollary}[theorem]{Corollary}
\newtheorem{conjecture}[theorem]{Conjecture}

\theoremstyle{definition}
\newtheorem{definition}[theorem]{Definition}
\newtheorem{remark}[theorem]{Remark}
\numberwithin{equation}{section}
\newtheorem{example}[theorem]{Example}
\newtheorem{assumption}[theorem]{Assumption}
\newtheorem{setting}[theorem]{Setting}

\usepackage{fancyhdr}
\begin{document}

\normalfont

\title{Period Rings with Big Coefficients and Applications III}
\author{Xin Tong} 

\maketitle

\begin{abstract}
\rm We continue our study on the corresponding period rings with big coefficients, with the corresponding application in mind on relative $p$-adic Hodge theory and noncommutative analytic geometry. In this article, we extend the discussion of the corresponding noncommutative descent over \'etale topology to the corresponding noncommutative descent over pro-\'etale topology in both Tate and analytic setting. 
\end{abstract}

\footnotetext[1]{Version: Mar 7 2021.}


\tableofcontents


\section{Introduction}

\subsection{Introduction and Main Results}

\noindent In our previous work on the corresponding period rings with big coefficients and applications, we essentially studied many very general deformations of the corresponding pseudocoherent sheaves over adic Banach rings in the sense of \cite{KL1} and \cite{KL2}. We studied the corresponding glueing in the fashion of \cite{KL1} and \cite{KL2} carrying sufficiently large coefficients. Although we have some conditions on both the corresponding adic spaces we are considering and the corresponding coefficients, but the descent results on their own are already general enough to tackle some specific situations in the Hodge-Iwasawa theoretic consideration and noncommutative analytic geometry such as in \cite{TX3} and \cite{TX4}.

\indent First in our current scope of discussion we will first consider the corresponding extension of our paper \cite{TX2} to the corresponding pro-\'etale topology setting. In the Tate setting, this will follow the corresponding unrelative situation in \cite{KL2} while in the analytic situation the corresponding will following the corresponding unrelative situation in \cite{Ked1} (in the context of \cref{chapter2} and \cref{chapter3}):

\begin{theorem} \mbox{\bf{(See \cref{proposition2.10}, \cref{proposition3.17})}}
Over sheafy Tate adic Banach rings or analytic Huber rings, the corresponding descent of stably pseudocoherent modules (with respect to the underlying spaces) carrying the corresponding noncommutative coefficients holds over \'etale sites and pro-\'etale sites. Here for pro-\'etale sites we will consider only perfectoid rings. 	
\end{theorem}

\indent We then consider some application to the corresponding 'complete' noetherian situation (this will mean that we consider the noetherian deformation of noetherian rings) in the context of \cref{section6}:

\begin{theorem} \mbox{\bf{(See \cref{proposition6.9})}} 
Over noetherian sheafy Tate adic Banach rings, the corresponding descent of finitely presented modules (with respect to the underlying spaces) carrying the corresponding noncommutative coefficients (such that the products of rings are noetherian as well) holds over \'etale sites.   	
\end{theorem}

\indent We also discussed the corresponding possible application to the descent in certain cases over the corresponding $\infty$-analytic stacks after Bambozzi-Ben-Bassat-Kremnizer \cite{BBBK} and Bambozzi-Kremnizer \cite{BK}, where the latter is really more related to the adic geometry we considered after \cite{KL1} and \cite{KL2}. Our approach certainly is definitely parallel to \cite{KL1} and \cite{KL2}, which is not actually parallel to more $\infty$-method considered by Ben-Bassat-Kremnizer in \cite{BBK}.

\subsection{Future Work}

\indent We expect more interesting applications. For instance partial noetherian spaces could admit some desired descent carrying general noncommutative coefficients (certainly not need to be noetherian coefficients), such as different types of eigenvarieties or Shimura varieties.


\section{Descent over Analytic Adic Banach Rings in the Rational Setting} \label{chapter2}

\subsection{Noncommutative Functional-Analytic Pseudocoherence over Pro-\'Etale Topology}

\indent Now we consider the corresponding discussion of the corresponding glueing deformed pseudocoherent sheaves over \'etale topology which generalize the corresponding discussion in \cite{Ked1}. 

\begin{setting}
Let $(W,W^+)$ be Tate adic Banach uniform ring which is defined over $\mathbb{Q}_p$. And we assume that the ring $V$ over $\mathbb{Q}_p$ is a Banach ring over the $\mathbb{Q}_p$. Assume that the ring $W$ is sheafy. 	
\end{setting}

\begin{remark}
The corresponding glueing of the corresponding pseudocoherent sheaves over the corresponding pro-\'etale site does not need the corresponding notions on the corresponding stability beyond the corresponding \'etale-stably pseudocoherence.	
\end{remark}

\begin{lemma} \mbox{\bf{(After Kedlaya-Liu \cite[Proposition 3.4.3]{KL2})}}
Suppose that we are taking $(W,W^+)$ to be Fontaine perfectoid (note that we are considering the corresponding context of Tate adic Banach situation). Then we have that over the corresponding \'etale site, the corresponding group $H^i(\mathrm{Spa}(W,W^+)_\text{p\'et},\mathcal{O}\widehat{\otimes}V)$ vanishes for each $i>0$, while we have that the corresponding group $H^0(\mathrm{Spa}(W,W^+)_\text{p\'et},\mathcal{O}\widehat{\otimes}V)$ is just isomorphic $W\widehat{\otimes}V$.	
\end{lemma}

\begin{proof}
See \cite[Proposition 3.4.3]{KL2}, where the corresponding \cite[Corollary 3.3.20]{KL2} applies.	
\end{proof}

\indent Now we proceed to consider the corresponding pro-\'etale topology. First we recall the following result from \cite[Lemma 3.4.4]{KL2} which does hold in our situation since we did not change much on the corresponding underlying adic spaces.

\begin{lemma}  \mbox{\bf{(Kedlaya-Liu \cite[Lemma 3.4.4]{KL2})}}
Consider the ring $W$ as above which is furthermore assumed to be Fontaine perfectoid in the sense of \cite[Definition 3.3.1]{KL2}, and consider any direct system of faithfully finite \'etale morphisms as:
\[
\xymatrix@C+0pc@R+0pc{
W_0 \ar[r] \ar[r] \ar[r] &W_1 \ar[r] \ar[r] \ar[r] &W_2  \ar[r] \ar[r] \ar[r] &...,   
}
\]	
where $A_0$ is just the corresponding base $A$. Then the corresponding completion of the infinite level could be decomposed as:
\begin{align}
W_0\oplus\widehat{\bigoplus}_{k=0}^\infty W_k/W_{k-1}.	
\end{align}
As in \cite[Lemma 3.4.4]{KL2} we endow the corresponding completion mentioned above with the corresponding seminorm spectral. And for the corresponding quotient we endow with the corresponding quotient norm.
\end{lemma}


\indent Then we consider the corresponding deformation by the ring $V$ over $\mathbb{Q}_p$:

\begin{lemma}  \mbox{\bf{(After Kedlaya-Liu \cite[Lemma 3.4.4]{KL2})}}
Consider the ring $A$ as above which is furthermore assumed to be Fontaine perfectoid in the sense of \cite[Definition 3.3.1]{KL2}, and consider any direct system of faithfully finite \'etale morphisms as:
\[
\xymatrix@C+0pc@R+0pc{
W_{0,V} \ar[r] \ar[r] \ar[r] &W_{1,V} \ar[r] \ar[r] \ar[r] &W_{2,V}  \ar[r] \ar[r] \ar[r] &..., 
}  
\]	
where $W_0$ is just the corresponding base $W$. Then the corresponding completion of the infinite level could be decomposed as:
\begin{align}
W_{0,V}\oplus\widehat{\bigoplus}_{k=0}^\infty W_{k,V}/W_{k-1,V}.	
\end{align}
\end{lemma}

\

\begin{corollary}\mbox{\bf{(After Kedlaya-Liu \cite[Corollary 3.4.5]{KL2})}} 
Starting with an adic Banach ring which is as in \cite[Corollary 3.4.5]{KL2} assumed to be Fontaine perfectoid in the sense of \cite[Definition 3.3.1]{KL2}. And as in the previous two lemmas we consider an admissible infinite direct system:
\[
\xymatrix@C+0pc@R+0pc{
W_{0,V} \ar[r] \ar[r] \ar[r] &W_{1,V} \ar[r] \ar[r] \ar[r] &W_{2,V}  \ar[r] \ar[r] \ar[r] &...,   
}
\]
whose infinite level will be assumed to take the corresponding spectral seminorm as in the \cite[Corollary 3.4.5]{KL2}. Then carrying the corresponding coefficient $V$, we have in such situation the corresponding 2-pseudoflatness of the corresponding embedding map:
\begin{align}
W_{0,V}\rightarrow W_{\infty,V}.
\end{align}
Here $W_{\infty,V}$ denotes the corresponding completion of the limit $\varinjlim_{k\rightarrow\infty}W_{k,V}$.
\end{corollary}

\begin{proof}
See \cite[Corollary 3.4.5]{KL2}.	
\end{proof}

\indent Now recall that from \cite{KL2} since we do not have to modify the corresponding underlying spatial context, so we will also only have to consider the corresponding stability with respect to \'etale topology even although we are considering the corresponding profinite \'etale site in our current section. We first generalize the corresponding Tate acyclicity in \cite{KL2} to the corresponding $V$-relative situation in our current situation:


\begin{remark}
In the following we only consider the corresponding perfectoid ring $W$.	
\end{remark}

\begin{proposition}\mbox{\bf{(After Kedlaya-Liu \cite[Theorem 3.4.6]{KL2})}} \label{proposition2.7} 
Now we consider the corresponding pro-\'etale site of $X$ which we denote it by $X_\text{p\'et}$. We assume we have a stable basis $\mathbb{H}$ of the corresponding perfectoid subdomains for this profinite \'etale site where each morphism therein will be \'etale pseudoflat. We consider as in \cite[Theorem 3.4.6]{KL2} a corresponding module over $W\widehat{\otimes}V$ which is assumed to be \'etale-stably pseudocoherent. Then consider the corresponding presheaf $\widetilde{G}$ attached to this \'etale-stably pseudocoherent module with respect to in our situation the corresponding chosen well-defined basis $\mathbb{H}$, we have that the corresponding sheaf over some element $Y\in \mathbb{H}$ (that is to say over $\mathcal{O}_{Y_\text{p\'et}}\widehat{\otimes}V$) is acyclic and is acyclic with respect to some \v{C}eck covering coming from elements in $\mathbb{H}$. 
\end{proposition}

\begin{proof}
The corresponding proof could be made as in the corresponding nonrelative setting as in \cite[Theorem 3.4.6]{KL2}. As in \cite[Theorem 3.4.6]{KL2} the corresponding first step is to check the statement is true with respect to some specific basis consisting of all the perfectoid subdomains which are faithfully finite \'etale over some specific perfectoid subdomain $Y$ (along towers). Then the corresponding statement could be reduced to those with respect to covering which is simple Laurent rational and the corresponding covering in the sense of the previous corollary. Then actually in this first scope of consideration, the statement holds by previous corollary and \cite[Theorem 2.12]{TX2}. Then in general for some general $Y$ we consider the corresponding covering of $Y$ consisting of elements satisfying the previous situation, then apply the previous argument to the corresponding basis in current situation satisfying the previous situation Then we consider the corresponding sheafification of the following chosen short exact sequence:
\[
\xymatrix@C+0pc@R+0pc{
0 \ar[r] \ar[r] \ar[r] &G'' \ar[r] \ar[r] \ar[r] &G'  \ar[r] \ar[r] \ar[r] &G \ar[r] \ar[r] \ar[r] &0 ,   
}
\]
with $G'$ finite projective. Then the idea is to consider the corresponding diagram factoring through the corresponding pro-\'etale sheafification of this short exact sequence down to the corresponding \'etale site. Then the five lemma will finish the corresponding proof as one considers the following commutative diagram:
\[
\xymatrix@C+0pc@R+3pc{
&G''_{V_Y} \ar[r] \ar[r] \ar[r] \ar[d] \ar[d] \ar[d] &G'_{V_Y}  \ar[r] \ar[r] \ar[r] \ar[d] \ar[d] \ar[d] &G_{V_Y} \ar[r] \ar[r] \ar[r] \ar[d] \ar[d] \ar[d] &0 ,  \\
0 \ar[r] \ar[r] \ar[r] &\Gamma(Y_\text{\'et},\widetilde{G''}) \ar[r] \ar[r] \ar[r] &\Gamma(Y_\text{\'et},\widetilde{G'})  \ar[r] \ar[r] \ar[r] &\Gamma(Y_\text{\'et},\widetilde{G}) \ar[r] \ar[r] \ar[r] &0.  \\ 
}
\]
Here the corresponding ring $V_Y$ is the adic Banach ring such that $Y=\mathrm{Spa}(V_Y,V_Y^+)$.
\end{proof}

\begin{definition}\mbox{\bf{(After Kedlaya-Liu \cite[Definition 3.4.7]{KL2})}} 
Consider the pro-\'etale site of $X$ attached to the adic Banach ring $(W,W^+)$. We will define a sheaf of module $G$ over $\widehat{\mathcal{O}}_\text{p\'et}\widehat{\otimes}V$ to be $V$-pseudocoherent if locally we can define this as a sheaf attached to a $V$-\'etale-stably pseudocoherent module. As in \cite[Definition 3.4.7]{KL2}, we do not have to consider the corresponding notion of $V$-profinite-\'etale-stably pseudocoherent module.	
\end{definition}

\begin{proposition}\mbox{\bf{(After Kedlaya-Liu \cite[Theorem 3.4.8]{KL2})}} \label{proposition2.10}
Taking the corresponding global section will realize the corresponding equivalence between the following two categories. The first one is the corresponding one of all the $V$-pseudocoherent sheaves over $\widehat{\mathcal{O}}_\text{p\'et}\widehat{\otimes}V$, while the second one is the corresponding one of all the $V$-\'etale-stably pseudocoherent modules over $W\widehat{\otimes}V$.
	
\end{proposition}

\begin{proof}
As in \cite[Theorem 3.4.8]{KL2}, the corresponding \cite[Proposition 9.2.6]{KL1} applies in the way that the corresponding conditions of \cite[Proposition 9.2.6]{KL1} hold in our current situation.	
\end{proof}

\indent Obviously we have the following analog of \cite[Corollary 3.4.9]{KL2}:


\begin{corollary} \mbox{\bf{(After Kedlaya-Liu \cite[Corollary 3.4.9]{KL2})}}
The following two categories are equivalent. The first is the corresponding category of all $V$-pseudocoherent sheaves over $\widehat{\mathcal{O}}_\text{p\'et}\widehat{\otimes}V$. The second is the corresponding category of all $V$-pseudocoherent sheaves over ${\mathcal{O}}_\text{\'et}\widehat{\otimes}V$. The corresponding functor is the corresponding pullback along the corresponding morphism of sites $X_{\text{p\'et}}\rightarrow X_{\text{\'et}}$. 	
\end{corollary}

\
%
%
%


\section{Descent over Analytic Huber Rings in the Integral Setting} \label{chapter3}

\subsection{Noncommutative Topological Pseudocoherence over \'Etale Topology}

\indent Now we consider the corresponding discussion of the corresponding glueing deformed pseudocoherent sheaves over \'etale topology which generalize the corresponding discussion in \cite{Ked1}.  And we will consider the corresponding situation as assumed in the following setting:

\begin{setting}
Let $(W,W^+)$ be analytic Huber uniform pair which is defined over $\mathbb{Z}_p$. And we assume that the ring $V$ over $\mathbb{Z}_p$ is a topological ring (complete) and splitting over the $\mathbb{Z}_p$. We now fix a corresponding stable basis $\mathbb{H}$ for the corresponding \'etale site of the adic space $\mathrm{Spa}(W,W^+)$, locally consisting of the corresponding compositions of rational localizations and the corresponding finite \'etale morphisms. And as in \cite[Hypothesis 1.10.3]{Ked1} we need to assume that the corresponding basis is made up of the corresponding adic spectrum of sheafy rings. Assume the corresponding sheafiness of the Huber ring $W$.
\end{setting}

\begin{definition} \mbox{\bf{(After Kedlaya-Liu \cite[Definition 2.5.9]{KL2})}}
We define a $V$-\'etale stably pseudocoherent module over the corresponding Huber ring $W$ with respect to the corresponding basis $\mathbb{H}$ chosen above for the corresponding \'etale site of the analytic adic space $X$. We define a module over $W\widehat{\otimes}V$ is a $V$-\'etale stably pseudocoherent module if it is algebraically pseudocoherent (namely formed by the corresponding possibly infinite length resolution of finitely generated and projective modules) and at least complete with respect to the corresponding natural topology and also required to be also complete with respect to the natural topology along some base change with respect to any morphism in $\mathbb{H}$.
\end{definition}

\begin{definition} \mbox{\bf{(After Kedlaya-Liu \cite[Definition 2.5.9]{KL2})}}
Along the previous definition we have the corresponding notion of $V$-\'etale-pseudoflat left (or right respectively) modules with respect to the corresponding chosen basis $\mathbb{H}$. A such module is defined to be a topological module $G$ over $W\widehat{\otimes}V$ complete with respect to the natural topology and for any right (or left respectively) $V$-\'etale stably pseudocoherent module, they will jointly give the $\mathrm{Tor}_1$ vanishing.
\end{definition}

\begin{lemma}\mbox{\bf{(After Kedlaya-Liu \cite[Lemma 2.5.10]{KL2})}}
One could find another basis $\mathbb{H}'$ which is in our situation contained in the corresponding original basis $\mathbb{H}$ such that any morphism in $\mathbb{H}$ could be $V$-\'etale-pseudoflat with respect $\mathbb{H}'$ or just $\mathbb{H}$ itself.	
\end{lemma}

\begin{proof}
The derivation of such new basis is actually along the same way as in \cite[Lemma 2.5.10]{KL2} since the corresponding compositions of rational localizations and the finite \'etales are actually satisfying the corresponding conditions in the statement by \cite[Theorem 2.12]{TX2} over the analytic topology. In general one just show any general morphism will also decompose in the same way, which one can proves certainly as in \cite[Lemma 2.5.10]{KL2}, where one instead consider in the current context the corresponding basis spreading result in \cite[Lemma 1.10.4]{Ked1}. 
\end{proof}


\begin{proposition} \mbox{\bf{(After Kedlaya-Liu \cite[Theorem 2.5.11]{KL2})}}
Consider the site $X_\text{\'et}$ and consider the basis $\mathbb{H}$. Take any $V$-\'etale stably pseudocoherent module $M$ over $W\widehat{\otimes}V$. Consider the corresponding presheaf by taking the inverse limit throughout all the corresponding base change along morphisms in $\mathbb{H}$. Then we have the corresponding acyclicity of the presheaf over any element in $\mathbb{H}$ and any covering of this element by elements in $\mathbb{H}$.	
\end{proposition}

\begin{proof}
As in \cite[Theorem 2.5.11]{KL2} apply the corresponding \cite[Proposition 8.2.21]{KL1}.	
\end{proof}

\begin{proposition}\mbox{\bf{(After Kedlaya-Liu \cite[Lemma 2.5.13]{KL2})}}
The corresponding glueing of $V$-\'etale-stably pseudocoherent modules holds in this current situation over $W$ along binary morphisms (namely along binary rational decomposition). 	
\end{proposition}

\begin{proof}
We adapt the argument of \cite[Lemma 2.5.13]{KL2} to our current situation. Take the corresponding map to be $W\rightarrow W_1\bigoplus W_2$, and we denote the two spaces associated $W_1$ and $W_2$ by $Y_1$ and $Y_2$. Establish a corresponding descent datum of $V$-\'etale-stably pseudocoherent modules along this decomposition of $\mathrm{Spa}(W,W^+)$. One can realize this by acyclicity a sheaf $G$ over the whole space $\mathrm{Spa}(W,W^+)$ Then we have for each $k=1,2$ that the corresponding identification of $H^0(Y_{k,\text{\'et}},G)$ with $H^0(Y_k,G)$ and the vanishing of the $i$-th cohomology for higher $i$. By the corresponding results we have in the corresponding rational localization situation we vanishing of the corresponding $H^i(\mathrm{Spa}(W,W^+),G),i\geq 1$. Now consider a finite free covering $C$ of $G$ and take the kernel $K$ to form the corresponding exact sequence:
\[
\xymatrix@C+0pc@R+0pc{
0 \ar[r] \ar[r] \ar[r] &K \ar[r] \ar[r] \ar[r] &C  \ar[r] \ar[r] \ar[r] &G \ar[r] \ar[r] \ar[r] &0 ,   
}
\]
which will give rise to the following short exact sequence:
\[
\xymatrix@C+0pc@R+0pc{
0 \ar[r] \ar[r] \ar[r] &K(\mathrm{Spa}(W,W^+)) \ar[r] \ar[r] \ar[r] &C(\mathrm{Spa}(W,W^+))  \ar[r] \ar[r] \ar[r] &G(\mathrm{Spa}(W,W^+)) \ar[r] \ar[r] \ar[r] &0.   
}
\]
To finish we have also to show the corresponding resulting global section is actually $V$-\'etale-stably pseudocoherent. For this, we refer the readers to \cite[Lemma 2.22]{TX2}.

\end{proof}

\begin{proposition} \mbox{\bf{(After Kedlaya-Liu \cite[Theorem 2.5.14]{KL2})}}
Taking the corresponding global section will realize the equivalence between the following two categories. The first is the one of all the corresponding $V$-pseudocoherent sheaves over $\mathcal{O}_{\text{\'et}}\widehat{\otimes}V$. The second is the one of all the corresponding $V$-\'etale-stably pseudocoherent modules.	
\end{proposition}

\begin{proof}
This is by considering and applying the corresponding \cite[Lemma 1.10.4]{Ked1} in our current context. The refinement comes from the acyclicity, the transitivity is straightforward, the corresponding binary rational localization situation is the previous proposition, while the corresponding last condition will basically comes from the corresponding f.p. descent \cite[Chapitre VIII]{SGAI}.	
\end{proof}

\
\subsection{Noncommutative Topological Pseudocoherence over Pro-\'Etale Topology}

\begin{setting}
Let $(W,W^+)$ be analytic Huber ring which is defined over $\mathbb{Z}_p$. And we assume that the ring $V$ over $\mathbb{Z}_p$ is a complete topological ring over the $\mathbb{Z}_p$ which is completely exact. Assume that the ring $W$ is sheafy. We assume $p$ is a topologically nilpotent element.
\end{setting}

\begin{remark}
The corresponding glueing of the corresponding pseudocoherent sheaves over the corresponding pro-\'etale site does not need the corresponding notions on the corresponding stability beyond the corresponding \'etale-stably pseudocoherence.	
\end{remark}

\begin{lemma} \mbox{\bf{(After Kedlaya-Liu \cite[Proposition 3.4.3]{KL2})}}
Suppose that we are taking $(W,W^+)$ to be perfectoid now in the sense of \cite[Definition 2.1.1]{Ked1} (note that we are considering the corresponding context of analytic Huber pair situation). Then we have that over the corresponding \'etale site, the corresponding group $H^i(\mathrm{Spa}(W,W^+)_\text{p\'et},\mathcal{O}\widehat{\otimes}V)$ vanishes for each $i>0$, while we have that the corresponding group $H^0(\mathrm{Spa}(W,W^+)_\text{p\'et},\mathcal{O}\widehat{\otimes}V)$ is just isomorphic $W\widehat{\otimes}V$.	
\end{lemma}

\begin{proof}
See \cite[Proposition 3.4.3]{KL2}.	
\end{proof}

\indent Now we proceed to consider the corresponding pro-\'etale topology. First we recall the following result from \cite[Lemma 3.4.4]{KL2} which does hold in our situation since we did not change much on the corresponding underlying adic spaces.

\begin{lemma}  \mbox{\bf{(Kedlaya-Liu \cite[Lemma 3.4.4]{KL2})}}
Consider the ring $W$ as above which is furthermore assumed to be perfectoid in the sense of instead \cite[Definition 2.1.1]{Ked1}, and consider any direct system of faithfully finite \'etale morphisms as:
\[
\xymatrix@C+0pc@R+0pc{
W_0 \ar[r] \ar[r] \ar[r] &W_1 \ar[r] \ar[r] \ar[r] &W_2  \ar[r] \ar[r] \ar[r] &...,   
}
\]	
where $A_0$ is just the corresponding base $A$. Then the corresponding completion of the infinite level could be decomposed as:
\begin{align}
W_0\oplus\widehat{\bigoplus}_{k=0}^\infty W_k/W_{k-1}.	
\end{align}
As in \cite[Lemma 3.4.4]{KL2} we endow the corresponding completion mentioned above with the corresponding seminorm spectral. And for the corresponding quotient we endow with the corresponding quotient norm.
\end{lemma}

\indent Then we consider the corresponding deformation by the ring $V$ over $\mathbb{Z}_p$:

\begin{lemma}  \mbox{\bf{(After Kedlaya-Liu \cite[Lemma 3.4.4]{KL2})}}
Consider the ring $A$ as above which is furthermore assumed to be perfectoid in the sense of \cite[Definition 2.1.1]{Ked1}, and consider any direct system of faithfully finite \'etale morphisms as:
\[
\xymatrix@C+0pc@R+0pc{
W_{0,V} \ar[r] \ar[r] \ar[r] &W_{1,V} \ar[r] \ar[r] \ar[r] &W_{2,V}  \ar[r] \ar[r] \ar[r] &..., 
}  
\]	
where $W_0$ is just the corresponding base $W$. Then the corresponding completion of the infinite level could be decomposed as:
\begin{align}
W_{0,V}\oplus\widehat{\bigoplus}_{k=0}^\infty W_{k,V}/W_{k-1,V}.	
\end{align}
\end{lemma}

\

\begin{corollary}\mbox{\bf{(After Kedlaya-Liu \cite[Corollary 3.4.5]{KL2})}} 
Starting with an analytic Huber ring which is as in \cite[Corollary 3.4.5]{KL2} assumed to be perfectoid in the sense of \cite[Definition 2.1.1]{Ked1}. And as in the previous two lemmas we consider an admissible infinite direct system:
\[
\xymatrix@C+0pc@R+0pc{
W_{0,V} \ar[r] \ar[r] \ar[r] &W_{1,V} \ar[r] \ar[r] \ar[r] &W_{2,V}  \ar[r] \ar[r] \ar[r] &...,   
}
\]
whose infinite level will be assumed to take the corresponding spectral seminorm as in the \cite[Corollary 3.4.5]{KL2}. Then carrying the corresponding coefficient $V$, we have in such situation the corresponding 2-pseudoflatness of the corresponding embedding map:
\begin{align}
W_{0,V}\rightarrow W_{\infty,V}.
\end{align}
Here $W_{\infty,V}$ denotes the corresponding completion of the limit $\varinjlim_{k\rightarrow\infty}W_{k,V}$.
\end{corollary}

\begin{proof}
See \cite[Corollary 3.4.5]{KL2}.	
\end{proof}

\indent Now recall that from \cite{KL2} since we do not have to modify the corresponding underlying spatial context, so we will also only have to consider the corresponding stability with respect to \'etale topology even although we are considering the corresponding profinite \'etale site in our current section. We first generalize the corresponding Tate acyclicity in \cite{KL2} to the corresponding $V$-relative situation in our current situation:

\begin{remark}
In the following we only consider the corresponding perfectoid ring $W$.	
\end{remark}

\begin{proposition}\mbox{\bf{(After Kedlaya-Liu \cite[Theorem 3.4.6]{KL2})}} 
Now we consider the corresponding pro-\'etale site of $X$ which we denote it by $X_\text{p\'et}$. Then we assume we have a stable basis $\mathbb{H}$ of the corresponding perfectoid subdomains for this profinite \'etale site where each morphism therein will be \'etale pseudoflat. Then we consider as in \cite[Theorem 3.4.6]{KL2} a corresponding module over $W\widehat{\otimes}V$ which is assumed to be \'etale-stably pseudocoherent. Then consider the corresponding presheaf $\widetilde{G}$ attached to this \'etale-stably pseudocoherent with respect to in our situation the corresponding chosen well-defined basis $\mathbb{H}$, we have that the corresponding sheaf over some element $Y\in \mathbb{H}$ (that is to say over $\mathcal{O}_{Y_\text{p\'et}}\widehat{\otimes}V$) is acyclic and is acyclic with respect to some \v{C}eck covering coming from elements in $\mathbb{H}$. 
\end{proposition}

\begin{proof}
See \cref{proposition2.7}.
\end{proof}

\begin{definition}\mbox{\bf{(After Kedlaya-Liu \cite[Definition 3.4.7]{KL2})}} 
Consider the pro-\'etale site of $X$ attached to the adic Banach ring $(W,W^+)$. We will define a sheaf of module $G$ over $\widehat{\mathcal{O}}_\text{p\'et}\widehat{\otimes}V$ to be $V$-pseudocoherent if locally we can define this as a sheaf attached to a $V$-\'etale-stably pseudocoherent module. As in \cite[Definition 3.4.7]{KL2}, we do not have to consider the corresponding notion of $V$-profinite-\'etale-stably pseudocoherent module.	
\end{definition}

\begin{proposition}\mbox{\bf{(After Kedlaya-Liu \cite[Theorem 3.4.8]{KL2})}} \label{proposition3.17}
Taking the corresponding global section will realize the corresponding equivalence between the following two categories. The first one is the corresponding one of all the $V$-pseudocoherent sheaves over $\widehat{\mathcal{O}}_\text{p\'et}\widehat{\otimes}V$, while the second one is the corresponding one of all the $V$-\'etale-stably pseudocoherent module over $W\widehat{\otimes}V$.
	
\end{proposition}

\begin{proof}
As in \cite[Theorem 3.4.8]{KL2}, the corresponding analog of \cite[Proposition 9.2.6]{KL1} applies in the way that the corresponding conditions of the analog of \cite[Proposition 9.2.6]{KL1} hold in our current situation. Also see \cite[Theorem 2.9.9, Remark 2.9.10 and Lemma 1.10.4]{Ked1}.
\end{proof}

\indent Obviously we have the following analog of \cite[Corollary 3.4.9]{KL2}:

\begin{corollary} \mbox{\bf{(After Kedlaya-Liu \cite[Corollary 3.4.9]{KL2})}}
The following two categories are equivalent. The first is the corresponding category of all $V$-pseudocoherent sheaves over $\widehat{\mathcal{O}}_\text{p\'et}\widehat{\otimes}V$. The second is the corresponding category of all $V$-pseudocoherent sheaves over ${\mathcal{O}}_\text{\'et}\widehat{\otimes}V$. The corresponding functor is the corresponding pullback along the corresponding morphism of sites $X_{\text{p\'et}}\rightarrow X_{\text{\'et}}$. 	
\end{corollary}

\


%
%
%
%
%
%
%
%
%
%
%
%
%
%
%
%


\section{Descent in General Setting for Analytic Huber Pairs}

\noindent We now consider the situation where there is no condition on the base (except that we will assume certainly the corresponding sheafiness). We will discuss along \cite{Ked1} the corresponding descent in analytic and \'etale topology, as well as the corresponding quasi-Stein spaces. The current situation will not assume we are working over $\mathbb{Z}_p$.

\begin{remark}
The corresponding discussion in this section and in the following section is actually an exercise proposed by Kedlaya included in \cite{Ked1}, where the corresponding \'etale situation essentially is proposed by Kedlaya in \cite[Discussion after Lemma 1.10.4]{Ked1} while we give the corresponding exposition in both \'etale and pro-\'etale situations closely after \cite{KL1} and \cite{KL2}.	
\end{remark}

\subsection{\'Etale Topology}

\begin{setting}
Let $(W,W^+)$ be analytic Huber uniform pair. We now fix a corresponding stable basis $\mathbb{H}$ for the corresponding \'etale site of the adic space $\mathrm{Spa}(W,W^+)$, locally consisting of the corresponding compositions of rational localizations and the corresponding finite \'etale morphisms. And as in \cite[Hypothesis 1.10.3]{Ked1} we need to assume that the corresponding basis is made up of the corresponding adic spectrum of sheafy rings. Assume the corresponding sheafiness of the Huber ring $W$.
\end{setting}

\begin{remark}
One should treat our discussion in this section as essentially some detailization of some exercise predicted in \cite{Ked1}.	
\end{remark}

\begin{definition} \mbox{\bf{(After Kedlaya-Liu \cite[Definition 2.5.9]{KL2})}}
We define the \'etale stably pseudocoherent module over the corresponding Huber ring $W$ with respect to the corresponding basis $\mathbb{H}$ chosen above for the corresponding \'etale site of the analytic adic space $X$. We define a module over $W$ to be an \'etale stably pseudocoherent module if it is algebraically pseudocoherent (namely formed by the corresponding possibly infinite length resolution of finitely generated and projective modules) and at least complete with respect to the corresponding natural topology and also required to be also complete with respect to the natural topology along some base change with respect to any morphism in $\mathbb{H}$.
\end{definition}

\begin{definition} \mbox{\bf{(After Kedlaya-Liu \cite[Definition 2.5.9]{KL2})}}
Along the previous definition we have the corresponding notion of \'etale-pseudoflat left (or right respectively) modules with respect to the corresponding chosen basis $\mathbb{H}$. A such module is defined to be a topological module $G$ over $W$ complete with respect to the natural topology and for any right (or left respectively) \'etale stably pseudocoherent module, they will jointly give the $\mathrm{Tor}_1$ vanishing.
\end{definition}

\begin{remark}
Although in this definition we considered the corresponding left or right modules, but essentially speaking this is no different from the usual situation since we are in the corresponding essential site theoretic situation.	
\end{remark}

\begin{lemma}\mbox{\bf{(After Kedlaya-Liu \cite[Lemma 2.5.10]{KL2})}}\label{lemma4.4}
One could find another basis $\mathbb{H}'$ which is in our situation contained in the corresponding original basis $\mathbb{H}$ such that any morphism in $\mathbb{H}$ could be \'etale-pseudoflat with respect $\mathbb{H}'$ or just $\mathbb{H}$ itself.	
\end{lemma}

\begin{proof}
The derivation of such new basis is actually along the same way as in \cite[Lemma 2.5.10]{KL2} since the corresponding compositions of rational localizations and the finite \'etales are actually satisfying the corresponding conditions in the statement by \cite[Theorem 2.12]{TX2} over the analytic topology. In general one just show any general morphism will also decompose in the same way, which one can proves certainly as in \cite[Lemma 2.5.10]{KL2}, where one instead consider in the current context the corresponding basis spreading result in \cite[Lemma 1.10.4]{Ked1}. 
\end{proof}

\begin{proposition} \mbox{\bf{(After Kedlaya-Liu \cite[Theorem 2.5.11]{KL2})}}
Consider the site $X_\text{\'et}$ and consider the basis $\mathbb{H}$. Take any \'etale stably pseudocoherent module $M$ over $W$. Consider the corresponding presheaf by taking the inverse limit throughout all the corresponding base change along morphisms in $\mathbb{H}$. Then we have the corresponding acyclicity of the presheaf over any element in $\mathbb{H}$ and any covering of this element by elements in $\mathbb{H}$.	
\end{proposition}

\begin{proof}
As in \cite[Theorem 2.5.11]{KL2} apply the corresponding \cite[Proposition 8.2.21]{KL1}.	
\end{proof}

\begin{proposition}\mbox{\bf{(After Kedlaya-Liu \cite[Lemma 2.5.13]{KL2})}}
The corresponding glueing of \'etale-stably pseudocoherent modules holds in this current situation over $W$ along binary morphisms (namely along binary rational decomposition). 	
\end{proposition}

\begin{proof}
We adapt the argument of \cite[Lemma 2.5.13]{KL2} to our current situation. Take the corresponding map to be $W\rightarrow W_1\bigoplus W_2$, and we denote the two spaces associated $W_1$ and $W_2$ by $Y_1$ and $Y_2$. Establish a corresponding descent datum of $V$-\'etale-stably pseudocoherent modules along this decomposition of $\mathrm{Spa}(W,W^+)$. One can realize this by acyclicity a sheaf $G$ over the whole space $\mathrm{Spa}(W,W^+)$ Then we have for each $k=1,2$ that the corresponding identification of $H^0(Y_{k,\text{\'et}},G)$ with $H^0(Y_k,G)$ and the vanishing of the $i$-th cohomology for higher $i$. By the corresponding results we have in the corresponding rational localization situation we have vanishing of the corresponding $H^i(\mathrm{Spa}(W,W^+),G),i\geq 1$. Now consider a finite free covering $C$ of $G$ and take the kernel $K$ to form the corresponding exact sequence:
\[
\xymatrix@C+0pc@R+0pc{
0 \ar[r] \ar[r] \ar[r] &K \ar[r] \ar[r] \ar[r] &C  \ar[r] \ar[r] \ar[r] &G \ar[r] \ar[r] \ar[r] &0 ,   
}
\]
which will give rise to the following short exact sequence:
\[
\xymatrix@C+0pc@R+0pc{
0 \ar[r] \ar[r] \ar[r] &K(\mathrm{Spa}(W,W^+)) \ar[r] \ar[r] \ar[r] &C(\mathrm{Spa}(W,W^+))  \ar[r] \ar[r] \ar[r] &G(\mathrm{Spa}(W,W^+)) \ar[r] \ar[r] \ar[r] &0.   
}
\]
To finish we have also to show the corresponding resulting global section is actually \'etale-stably pseudocoherent. For this, we refer the readers to \cite[Lemma 2.22]{TX2}.

\end{proof}

\begin{proposition} \mbox{\bf{(Kedlaya \cite[Below Lemma 1.10.4]{Ked1}, After Kedlaya-Liu \cite[Theorem 2.5.14]{KL2})}}
Taking the corresponding global section will realize the equivalence between the following two categories. The first is the one of all the corresponding pseudocoherent sheaves over $\mathcal{O}_{\text{\'et}}$. The second is the one of all the corresponding \'etale-stably pseudocoherent modules.	
\end{proposition}

\begin{proof}
This is by considering and applying the corresponding \cite[Lemma 1.10.4]{Ked1} in our current context. The refinement comes from the acyclicity, the transitivity is straightforward, the corresponding binary rational localization situation is the previous proposition, while the corresponding last condition will basically come from the corresponding f.p. descent \cite[Chapitre VIII]{SGAI}.	
\end{proof}

\
\subsection{Pro-\'etale Topology}

\begin{setting}
Let $(W,W^+)$ be an analytic Huber ring. Assume that the ring $W$ is sheafy. We assume $p$ is a topologically nilpotent element.	
\end{setting}

\begin{remark}
Although it seems that the current discussion is basically outside the corresponding deformed setting, but the foundation here may be generalized to the deformed setting by taking the corresponding completed tensor product over $\mathbb{F}_1$ as in \cite{BBBK}.	
\end{remark}

\begin{remark}
The corresponding glueing of the corresponding pseudocoherent sheaves over the corresponding pro-\'etale site does not need the corresponding notions on the corresponding stability beyond the corresponding \'etale-stably pseudocoherence.	
\end{remark}

\begin{lemma} \mbox{\bf{(After Kedlaya-Liu \cite[Proposition 3.4.3]{KL2})}}
Suppose that we are taking $(W,W^+)$ to be perfectoid now in the sense of \cite[Definition 2.1.1]{Ked1} (note that we are considering the corresponding context of analytic Huber pair situation). Then we have that over the corresponding \'etale site, the corresponding group $H^i(\mathrm{Spa}(W,W^+)_\text{p\'et},\mathcal{O})$ vanishes for each $i>0$, while we have that the corresponding group $H^0(\mathrm{Spa}(W,W^+)_\text{p\'et},\mathcal{O})$ is just isomorphic to $W$.	
\end{lemma}

\begin{proof}
See \cite[Proposition 3.4.3]{KL2}.	
\end{proof}

\indent Now we proceed to consider the corresponding pro-\'etale topology. First we recall the following result from \cite[Lemma 3.4.4]{KL2} which does hold in our situation since we did not change much on the corresponding underlying adic spaces.\\

\begin{lemma}  \mbox{\bf{(Kedlaya-Liu \cite[Lemma 3.4.4]{KL2})}}
Consider the ring $W$ as above which is furthermore assumed to be perfectoid in the sense of instead \cite[Definition 2.1.1]{Ked1}, and consider any direct system of faithfully finite \'etale morphisms as:
\[
\xymatrix@C+0pc@R+0pc{
W_0 \ar[r] \ar[r] \ar[r] &W_1 \ar[r] \ar[r] \ar[r] &W_2  \ar[r] \ar[r] \ar[r] &...,   
}
\]	
where $W_0$ is just the corresponding base $W$. Then the corresponding completion of the infinite level could be decomposed as:
\begin{align}
W_0\oplus\widehat{\bigoplus}_{k=0}^\infty W_k/W_{k-1}.	
\end{align}
As in \cite[Lemma 3.4.4]{KL2} we endow the corresponding completion mentioned above with the corresponding seminorm spectral. And for the corresponding quotient we endow with the corresponding quotient norm.
\end{lemma}

\

\begin{corollary}\mbox{\bf{(After Kedlaya-Liu \cite[Corollary 3.4.5]{KL2})}} 
Starting with an analytic Huber ring which is as in \cite[Corollary 3.4.5]{KL2} assumed to be perfectoid in the sense of \cite[Definition 2.1.1]{Ked1}. And as in the previous two lemmas we consider an admissible infinite direct system:
\[
\xymatrix@C+0pc@R+0pc{
W_{0} \ar[r] \ar[r] \ar[r] &W_{1} \ar[r] \ar[r] \ar[r] &W_{2}  \ar[r] \ar[r] \ar[r] &...,   
}
\]
whose infinite level will be assumed to take the corresponding spectral seminorm as in the \cite[Corollary 3.4.5]{KL2}. Then we have in such situation the corresponding 2-pseudoflatness of the corresponding embedding map:
\begin{align}
W_{0}\rightarrow W_{\infty}.
\end{align}
Here $W_{\infty}$ denotes the corresponding completion of the limit $\varinjlim_{k\rightarrow\infty}W_{k}$.
\end{corollary}

\begin{proof}
See \cite[Corollary 3.4.5]{KL2}.	
\end{proof}

\indent Now recall that from \cite{KL2} since we do not have to modify the corresponding underlying spatial context, so we will also only have to consider the corresponding stability with respect to \'etale topology even although we are considering the corresponding profinite \'etale site in our current section. We first generalize the corresponding Tate acyclicity in \cite{KL2} to the corresponding our current situation.\\

\begin{remark}
In the following we only consider the base spaces which are perfectoid.	
\end{remark}

\begin{proposition}\mbox{\bf{(After Kedlaya-Liu \cite[Theorem 3.4.6]{KL2})}} 
Now we consider the corresponding pro-\'etale site of $X$ which we denote it by $X_\text{p\'et}$. Then we assume we have a stable basis $\mathbb{H}$ of the corresponding perfectoid subdomains for this profinite \'etale site where each morphism therein will be \'etale pseudoflat. Then we consider as in \cite[Theorem 3.4.6]{KL2} a corresponding module over $W$ which is assumed to be \'etale-stably pseudocoherent. Then consider the corresponding presheaf $\widetilde{G}$ attached to this \'etale-stably pseudocoherent with respect to in our situation the corresponding chosen well-defined basis $\mathbb{H}$, we have that the corresponding sheaf over some element $Y\in \mathbb{H}$ (that is to say over $\mathcal{O}_{Y_\text{p\'et}}$) is acyclic and is acyclic with respect to some \v{C}eck covering coming from elements in $\mathbb{H}$. 
\end{proposition}

\begin{proof}
See \cref{proposition2.7}.
\end{proof}

\begin{definition}\mbox{\bf{(After Kedlaya-Liu \cite[Definition 3.4.7]{KL2})}} 
Consider the pro-\'etale site of $X$ attached to the analytic Huber pair $(W,W^+)$. We will define a sheaf of module $G$ over $\widehat{\mathcal{O}}_\text{p\'et}$ to be pseudocoherent if locally we can define this as a sheaf attached to an \'etale-stably pseudocoherent module. As in \cite[Definition 3.4.7]{KL2}, we do not have to consider the corresponding notion of profinite-\'etale-stably pseudocoherent modules.	
\end{definition}

\begin{proposition}\mbox{\bf{(Kedlaya \cite[Section 3.8]{Ked1}, after Kedlaya-Liu \cite[Theorem 3.4.8]{KL2})}} \label{proposition4.20}
Taking the corresponding global section will realize the corresponding equivalence between the following two categories. The first one is the corresponding one of all the pseudocoherent sheaves over $\widehat{\mathcal{O}}_\text{p\'et}$, while the second one is the corresponding one of all the \'etale-stably pseudocoherent modules over $W$.
	
\end{proposition}

\begin{proof}
As in \cite[Theorem 3.4.8]{KL2}, the corresponding analog of \cite[Proposition 9.2.6]{KL1} applies in the way that the corresponding conditions of the analog of \cite[Proposition 9.2.6]{KL1} hold in our current situation. Also see \cite[Theorem 2.9.9, Remark 2.9.10 and Lemma 1.10.4]{Ked1}.
\end{proof}

\indent Obviously we have the following analog of \cite[Corollary 3.4.9]{KL2}:

\begin{corollary} \mbox{\bf{(After Kedlaya-Liu \cite[Corollary 3.4.9]{KL2})}}
The following two categories are equivalent. The first is the corresponding category of all pseudocoherent sheaves over $\widehat{\mathcal{O}}_\text{p\'et}$. The second is the corresponding category of all pseudocoherent sheaves over ${\mathcal{O}}_\text{\'et}$. The corresponding functor is the corresponding pullback along the corresponding morphism of sites $X_{\text{p\'et}}\rightarrow X_{\text{\'et}}$. 	
\end{corollary}

\

\section{Descent in General Setting for Analytic Adic Banach Rings}

\noindent We now translate the results in previous section to the corresponding analytic adic Banach context. Again everything is essentially as proposed in \cite[Below Lemma 1.10.4, also see Section 3.8]{Ked1}.

\subsection{\'Etale Topology}

\begin{setting}
Let $(W,W^+)$ be analytic adic Banach ring. We now fix a corresponding stable basis $\mathbb{H}$ for the corresponding \'etale site of the adic space $\mathrm{Spa}(W,W^+)$, locally consisting of the corresponding compositions of rational localizations and the corresponding finite \'etale morphisms. And as in \cite[Hypothesis 1.10.3]{Ked1} we need to assume that the corresponding basis is made up of the corresponding adic spectrum of sheafy rings. Assume the corresponding sheafiness of the Huber ring $W$.
\end{setting}

\begin{definition} \mbox{\bf{(After Kedlaya-Liu \cite[Definition 2.5.9]{KL2})}}
We define the \'etale-stably pseudocoherent module over the corresponding adic Banach $W$ with respect to the corresponding basis $\mathbb{H}$ chosen above for the corresponding \'etale site of the analytic adic space $X$. We define a module over $W$ to be an \'etale stably pseudocoherent module if it is algebraically pseudocoherent (namely formed by the corresponding possibly infinite length resolution of finitely generated and projective modules) and at least complete with respect to the corresponding natural topology and also required to be also complete with respect to the natural topology along some base change with respect to any morphism in $\mathbb{H}$.
\end{definition}


\begin{definition} \mbox{\bf{(After Kedlaya-Liu \cite[Definition 2.5.9]{KL2})}}
Along the previous definition we have the corresponding notion of \'etale-pseudoflat left (or right respectively) modules with respect to the corresponding chosen basis $\mathbb{H}$. A such module is defined to be a topological module $G$ over $W$ complete with respect to the natural topology and for any right (or left respectively) \'etale stably pseudocoherent module, they will jointly give the $\mathrm{Tor}_1$ vanishing.
\end{definition}

\begin{lemma}\mbox{\bf{(After Kedlaya-Liu \cite[Lemma 2.5.10]{KL2})}}
One could find another basis $\mathbb{H}'$ which is in our situation contained in the corresponding original basis $\mathbb{H}$ such that any morphism in $\mathbb{H}$ could be \'etale-pseudoflat with respect $\mathbb{H}'$ or just $\mathbb{H}$ itself.	
\end{lemma}

\begin{proof}
See \cref{lemma4.4}.
\end{proof}

\begin{proposition} \mbox{\bf{(After Kedlaya-Liu \cite[Theorem 2.5.11]{KL2})}}
In our current analytic adic Banach situation, consider the site $X_\text{\'et}$ and consider the basis $\mathbb{H}$. Take any \'etale stably pseudocoherent module $M$ over $W$. Consider the corresponding presheaf by taking the inverse limit throughout all the corresponding base change along morphisms in $\mathbb{H}$. Then we have the corresponding acyclicity of the presheaf over any element in $\mathbb{H}$ and any covering of this element by elements in $\mathbb{H}$.	
\end{proposition}

\begin{proof}
As in \cite[Theorem 2.5.11]{KL2} apply the corresponding \cite[Proposition 8.2.21]{KL1}.	
\end{proof}

\begin{proposition}\mbox{\bf{(After Kedlaya-Liu \cite[Lemma 2.5.13]{KL2})}}
In our current analytic adic Banach situation, the corresponding glueing of \'etale-stably pseudocoherent modules holds in this current situation over $W$ along binary morphisms (namely along binary rational decomposition). 	
\end{proposition}

\begin{proof}
We adapt the argument of \cite[Lemma 2.5.13]{KL2} to our current situation over analytic adic Banach rings. Take the corresponding map to be $W\rightarrow W_1\bigoplus W_2$, and we denote the two spaces associated $W_1$ and $W_2$ by $Y_1$ and $Y_2$. Establish a corresponding descent datum of $V$-\'etale-stably pseudocoherent modules along this decomposition of $\mathrm{Spa}(W,W^+)$. One can realize this by acyclicity a sheaf $G$ over the whole space $\mathrm{Spa}(W,W^+)$ Then we have for each $k=1,2$ that the corresponding identification of $H^0(Y_{k,\text{\'et}},G)$ with $H^0(Y_k,G)$ and the vanishing of the $i$-th cohomology for higher $i$. By the corresponding results we have in the corresponding rational localization situation we vanishing of the corresponding $H^i(\mathrm{Spa}(W,W^+),G),i\geq 1$. Now consider a finite free covering $C$ of $G$ and take the kernel $K$ to form the corresponding exact sequence:
\[
\xymatrix@C+0pc@R+0pc{
0 \ar[r] \ar[r] \ar[r] &K \ar[r] \ar[r] \ar[r] &C  \ar[r] \ar[r] \ar[r] &G \ar[r] \ar[r] \ar[r] &0 ,   
}
\]
which will give rise to the following short exact sequence:
\[
\xymatrix@C+0pc@R+0pc{
0 \ar[r] \ar[r] \ar[r] &K(\mathrm{Spa}(W,W^+)) \ar[r] \ar[r] \ar[r] &C(\mathrm{Spa}(W,W^+))  \ar[r] \ar[r] \ar[r] &G(\mathrm{Spa}(W,W^+)) \ar[r] \ar[r] \ar[r] &0.   
}
\]
To finish we have also to show the corresponding resulting global section is actually \'etale-stably pseudocoherent. For this, we refer the readers to \cite[Lemma 2.22]{TX2}.

\end{proof}

\begin{proposition} \mbox{\bf{(After Kedlaya-Liu \cite[Theorem 2.5.14]{KL2})}}
In our current analytic adic Banach situation, taking the corresponding global section will realize the equivalence between the following two categories. The first is the one of all the corresponding pseudocoherent sheaves over $\mathcal{O}_{\text{\'et}}$. The second is the one of all the corresponding \'etale-stably pseudocoherent modules.	
\end{proposition}

\begin{proof}
This is by considering and applying the corresponding \cite[Lemma 1.10.4]{Ked1} in our current context. The refinement comes from the acyclicity, the transitivity is straightforward, the corresponding binary rational localization situation is the previous proposition, while the corresponding last condition will basically comes from the corresponding f.p. descent \cite[Chapitre VIII]{SGAI}.	
\end{proof}

\

\subsection{Pro-\'etale Topology}

\begin{setting}
Let $(W,W^+)$ be an analytic perfectoid adic Banach ring. Assume that the ring $W$ is sheafy.	We assume $p$ is a topologically nilpotent element.
\end{setting}

\begin{remark}
As in the corresponding analytic Huber ring situation, the corresponding glueing of the corresponding pseudocoherent sheaves over the corresponding pro-\'etale site does not need the corresponding notions on the corresponding stability beyond the corresponding \'etale-stably pseudocoherence.	
\end{remark}

\indent We translate the corresponding notions of perfectoid rings to the current context from \cite[Definition 2.1.1]{Ked1}:

\begin{definition}\mbox{\bf{(Kedlaya \cite[Definition 2.1.1]{Ked1})}}     \label{definition5.10}
Consider a general analytic adic Banach ring $(W,W^+)$, we will call it perfectoid if there exists a definition ideal $d\subset W^+$ such that $d^p$ contains $p$ and we have the corresponding Frobenius map from $W^+/d$ to $W^+/d^p$ is required to be surjective.	
\end{definition}

\begin{lemma} \mbox{\bf{(After Kedlaya-Liu \cite[Proposition 3.4.3]{KL2})}}
Suppose that we are taking $(W,W^+)$ to be perfectoid now in the sense of \cref{definition5.10} (note that we are considering the corresponding context of analytic adic Banach situation). Then we have that over the corresponding \'etale site, the corresponding group $H^i(\mathrm{Spa}(W,W^+)_\text{p\'et},\mathcal{O})$ vanishes for each $i>0$, while we have that the corresponding group $H^0(\mathrm{Spa}(W,W^+)_\text{p\'et},\mathcal{O})$ is just isomorphic $W$.	
\end{lemma}

\begin{proof}
See \cite[Proposition 3.4.3]{KL2}.	
\end{proof}

\indent Now we proceed to consider the corresponding pro-\'etale topology. First we recall the following result from \cite[Lemma 3.4.4]{KL2}.

\begin{lemma}  \mbox{\bf{(Kedlaya-Liu \cite[Lemma 3.4.4]{KL2})}}
Consider the ring $W$ as above which is furthermore assumed to be perfectoid in the sense of instead \cref{definition5.10}, and consider any direct system of faithfully finite \'etale morphisms as:
\[
\xymatrix@C+0pc@R+0pc{
W_0 \ar[r] \ar[r] \ar[r] &W_1 \ar[r] \ar[r] \ar[r] &W_2  \ar[r] \ar[r] \ar[r] &...,   
}
\]	
where $W_0$ is just the corresponding base $W$. Then the corresponding completion of the infinite level could be decomposed as:
\begin{align}
W_0\oplus\widehat{\bigoplus}_{k=0}^\infty W_k/W_{k-1}.	
\end{align}
As in \cite[Lemma 3.4.4]{KL2} we endow the corresponding completion mentioned above with the corresponding seminorm spectral. And for the corresponding quotient we endow with the corresponding quotient norm.
\end{lemma}

\

\begin{corollary}\mbox{\bf{(After Kedlaya-Liu \cite[Corollary 3.4.5]{KL2})}} 
Starting with an analytic adic Banach ring which is as in \cite[Corollary 3.4.5]{KL2} assumed to be perfectoid in the sense of \cref{definition5.10}. And as in the previous two lemmas we consider an admissible infinite direct system:
\[
\xymatrix@C+0pc@R+0pc{
W_{0} \ar[r] \ar[r] \ar[r] &W_{1} \ar[r] \ar[r] \ar[r] &W_{2}  \ar[r] \ar[r] \ar[r] &...,   
}
\]
whose infinite level will be assumed to take the corresponding spectral seminorm as in the \cite[Corollary 3.4.5]{KL2}. Then we have in such situation the corresponding 2-pseudoflatness of the corresponding embedding map:
\begin{align}
W_{0}\rightarrow W_{\infty}.
\end{align}
Here $W_{\infty}$ denotes the corresponding completion of the limit $\varinjlim_{k\rightarrow\infty}W_{k}$.
\end{corollary}

\begin{proof}
See \cite[Corollary 3.4.5]{KL2}.	
\end{proof}

\indent Now recall that from \cite{KL2} since we do not have modified the corresponding underlying spatial context, so we will also only have to consider the corresponding stability with respect to \'etale topology even although we are considering the corresponding profinite \'etale site in our current section. We first generalize the corresponding Tate acyclicity in \cite{KL2} to the corresponding our current situation:

\begin{proposition}\mbox{\bf{(After Kedlaya-Liu \cite[Theorem 3.4.6]{KL2})}} 
Now we consider the corresponding pro-\'etale site of $X$ which we denote it by $X_\text{p\'et}$. Then we assume we have a stable basis $\mathbb{H}$ of the corresponding perfectoid subdomains for this profinite \'etale site where each morphism therein will be \'etale pseudoflat. Then we consider as in \cite[Theorem 3.4.6]{KL2} a corresponding module over $W$ which is assumed to be \'etale-stably pseudocoherent. Then consider the corresponding presheaf $\widetilde{G}$ attached to this \'etale-stably pseudocoherent with respect to in our situation the corresponding chosen well-defined basis $\mathbb{H}$, we have that the corresponding sheaf over some element $Y\in \mathbb{H}$ (that is to say over $\mathcal{O}_{Y_\text{p\'et}}$) is acyclic and is acyclic with respect to some \v{C}eck covering coming from elements in $\mathbb{H}$. 
\end{proposition}

\begin{proof}
See \cref{proposition2.7}.
\end{proof}

\begin{definition}\mbox{\bf{(After Kedlaya-Liu \cite[Definition 3.4.7]{KL2})}} 
Consider the pro-\'etale site of $X$ attached to the adic Banach ring $(W,W^+)$. We will define a sheaf of module $G$ over $\widehat{\mathcal{O}}_\text{p\'et}$ to be pseudocoherent if locally we can define this as a sheaf attached to an \'etale-stably pseudocoherent module. As in \cite[Definition 3.4.7]{KL2}, we do not have to consider the corresponding notion of profinite-\'etale-stably pseudocoherent module.	
\end{definition}

\begin{proposition}\mbox{\bf{(After Kedlaya-Liu \cite[Theorem 3.4.8]{KL2})}}
Taking the corresponding global section will realize the corresponding equivalence between the following two categories. The first one is the corresponding one of all the pseudocoherent sheaves over $\widehat{\mathcal{O}}_\text{p\'et}$, while the second one is the corresponding one of all the \'etale-stably pseudocoherent modules over $W$.
	
\end{proposition}

\begin{proof}
As in \cite[Theorem 3.4.8]{KL2}, the corresponding analog of \cite[Proposition 9.2.6]{KL1} applies in the way that the corresponding conditions of the analog of \cite[Proposition 9.2.6]{KL1} hold in our current situation. Also see \cite[Theorem 2.9.9, Remark 2.9.10 and Lemma 1.10.4]{Ked1}. Certainly in this situation we need to translate the corresponding statemens in \cite[Theorem 2.9.9, Remark 2.9.10 and Lemma 1.10.4]{Ked1} to the corresponding adic Banach setting.	
\end{proof}

\indent Obviously we have the following analog of \cite[Corollary 3.4.9]{KL2}:

\begin{corollary} \mbox{\bf{(After Kedlaya-Liu \cite[Corollary 3.4.9]{KL2})}}
The following two categories are equivalent. The first is the corresponding category of all pseudocoherent sheaves over $\widehat{\mathcal{O}}_\text{p\'et}$. The second is the corresponding category of all pseudocoherent sheaves over ${\mathcal{O}}_\text{\'et}$. The corresponding functor is the corresponding pullback along the corresponding morphism of sites $X_{\text{p\'et}}\rightarrow X_{\text{\'et}}$. 	
\end{corollary}

\
\subsection{Quasi-Stein Spaces}

We now generalize our discussion in \cite{TX2} around pseudocoherent sheaves over quasi-Stein spaces to the general base situation without the corresponding restrictive assumption that we are working over $\mathbb{Z}_p$.

\begin{setting}
Assume that we are working over some quasi-Stein space $X$ which could be written as the corresponding direct limit of analytic adic Banach affinoids: $X:=\varinjlim_i  X_i$ such that $X_i$ could be written as the corresponding analytic affinoid $\mathrm{Spa}(W_i,W_i^+)$.	
\end{setting}


\begin{lemma}\mbox{\bf{(After Kedlaya-Liu \cite[Lemma 2.6.3]{KL2})}}
Consider a compatible family of Banach modules over the projective system $\mathcal{O}_{X_i}(X_i)$. Suppose that the corresponding transition map $p_i:\mathcal{O}_{X_{i}}(X_{i})\widehat{\otimes}_{\mathcal{O}_{X_{i+1}}(X_{i+1})} M_{i+1}\overset{}{\rightarrow}M_i$ is surjective. Then we have: I. The corresponding global section of $M$ throughout the limit $\mathcal{O}(X)=\varprojlim_i \mathcal{O}_{X_{i}}(X_{i})$ is dense in each section $M_i$ for any $i\geq 0$; II. The corresponding vanishing of $R^1\varprojlim_{i\rightarrow \infty}$ holds in our situation. (Parallel statement in analytic Huber ring situation holds as well.)  	
\end{lemma}

\begin{proof}
This is basically the general version of the corresponding result in \cite[Lemma 2.6.3]{KL2}. One could prove this in the same way. 
\end{proof}

\begin{lemma}\mbox{\bf{(After Kedlaya-Liu \cite[Corollary 2.6.4]{KL2})}}
Consider a compatible family of Banach modules over the projective system $\mathcal{O}_{X_i}(X_i)$. Suppose that each member in the family is stably-pseudocoherent. And suppose that the corresponding transition map $p_i:\mathcal{O}_{X_{i}}(X_{i})\widehat{\otimes}_{\mathcal{O}_{X_{i+1}}(X_{i+1})} M_{i+1}\overset{}{\rightarrow}M_i$ is strictly isomorphism. Then we have in our situation the corresponding projection from the global section $M=\varprojlim_i M_i$ to each $M_i$ (for each $i\geq 0$) is then isomorphism in our current situation as in \cite[Corollary 2.6.4]{KL2}. (Parallel statement in analytic Huber ring situation holds as well.)    	
\end{lemma}

\begin{proof}
This is basically the general version of the corresponding result in \cite[Corollary 2.6.4]{KL2}. One could prove this in the same way. 
\end{proof}

\begin{proposition}\mbox{\bf{(After Kedlaya-Liu \cite[Theorem 2.6.5]{KL2})}}
With the corresponding notations as above, we have that the global section of any pseudocoherent sheaf $M$ over $X$ will be dense in the section over any quasicompact. (Parallel statement in analytic Huber ring situation holds as well.)  	
\end{proposition}

\begin{proof}
See \cite[Theorem 2.6.5]{KL2}.	
\end{proof}

\begin{proposition}\mbox{\bf{(After Kedlaya-Liu \cite[Theorem 2.6.5]{KL2})}}
With the corresponding notations as above, we have that the global section of any pseudocoherent sheaf $M$ over $X$ will finitely generate each fiber $M_x$ over $\mathcal{O}_{X,x}$ for any $x\in X$. (Parallel statement in analytic Huber ring situation holds as well.)  
\end{proposition}

\begin{proof}
See \cite[Theorem 2.6.5]{KL2}.	
\end{proof}

\begin{proposition}\mbox{\bf{(After Kedlaya-Liu \cite[Theorem 2.6.5]{KL2})}}
With the corresponding notations as above, we have that the corresponding sheaf cohomology of any pseudocoherent sheaf $M$ over $X$ admits vanishing for degree bigger than 0. (Parallel statement in analytic Huber ring situation holds as well.)  
\end{proposition}

\begin{proof}
See \cite[Theorem 2.6.5]{KL2}.	
\end{proof}

\begin{proposition} \mbox{\bf{(After Kedlaya-Liu \cite[Corollary 2.6.6]{KL2})}}
For any pseudocoherent sheaf $M$ over the space $X$, we have in our current situation that the exact functor from the corresponding category of all the pseudocoherent sheaves to the corresponding one of all the corresponding $\mathcal{O}_X(X)$-modules through the corresponding taking the global section. (Parallel statement in analytic Huber ring situation holds as well.)  	
\end{proposition}

\begin{proposition}  \mbox{\bf{(After Kedlaya-Liu \cite[Corollary 2.6.8]{KL2})}}
For any pseudocoherent sheaf $M$ admitting a structure of vector bundles over the space $X$, we have in our current situation that the corresponding sufficient and necessary condition for the global section to be finite projective is exactly that the global section of $M$ is finitely generated. (Parallel statement in analytic Huber ring situation holds as well.)  	
\end{proposition}

\begin{proof}
See \cite[Corollary 2.6.8]{KL2}, it is straightforward to exact a global splitting from the local ones over quasi-compacts.	
\end{proof}

\begin{proposition}\mbox{\bf{(After Kedlaya-Liu \cite[Proposition 2.6.17]{KL2})}} 
Assume that the space $X$ is basically $m$-uniform in the sense of \cite{KL2}. Then we have the corresponding finiteness of the global section is equivalent to to the corresponding uniform finiteness through out all $X_i,i=0,1,2,...$. (Parallel statement in analytic Huber ring situation holds as well.)  	
\end{proposition}

\begin{proof}
See the proof given in \cite[Theorem 6.27]{TX2}. 	
\end{proof}

\


\section{Applications} \label{section6}

\subsection{Application to Noetherian Cases}

\indent The corresponding noetherian situation is expected to be in some sense better than just being pseudoflat with respect to the corresponding rational localization.

\begin{setting}
We now work with the corresponding analytic adic Banach rings over $\mathbb{Q}_p$ or $\mathbb{F}_p((t))$ as the corresponding base spaces. The deformation will happen along some Banach ring over $\mathbb{Q}_p$ or $\mathbb{F}_p((t))$. And moreover we have to assume that the corresponding noetherianness preserves under the corresponding deformation over some ring $V$. Namely we have to assume that $W\widehat{\otimes}_* V$ is always noetherian, for instance this could be achieved when we have that $V$ over $\mathbb{Q}_p$ comes from truncations of distribution algebra over some $p$-adic Lie groups by considering some further $p$-adic microlocal analysis.
\end{setting}

\begin{lemma}
With the notations in the \cite[Lemma 2.4.10]{KL2}, we have that the corresponding morphisms $W\widehat{\otimes}V\rightarrow B_1\widehat{\otimes}V$, $W\widehat{\otimes}V\rightarrow B_2\widehat{\otimes}V$ and $W\widehat{\otimes}V\rightarrow B_{12}\widehat{\otimes}V$ are 2-pseudoflat. And furthermore in our current context they are actually flat.
\end{lemma}

\begin{proof}
In our current situation, the corresponding 2-pseudoflatness is achieved since this is already true in the corresponding more general situation. However for finite presented modules (we do not have to consider the topological issues since we are in the noetherian situation) this is already flat which implies for finitely generated modules (we do not have to consider the topological issues since we are in the noetherian situation) this is already flat	as well. Then for any module which could be written as injective limit of finite ones, the result holds. 
\end{proof}

\begin{corollary}
The corresponding rational localization with respect to the adic Banach ring $(W,W^+)$ is flat (namely for all the corresponding finitely presented module over $W\widehat{\otimes}V$).	
\end{corollary}

\begin{proposition}
Then we consider the corresponding presheafification of any finitely generated module $M$ over $W\widehat{\otimes}V$, to be more precise over $\mathrm{Spa}(W,W^+)$ we will define the corresponding presheaf $\widetilde{M}$ by taking the inverse limit of the base changes of $M$ throughout all the rational localizations of $W$. Then we have that the Tate glueing property holds in our current situation for such presheaf $\widetilde{M}$.
\end{proposition}

\begin{proof}
This reduces to the previous lemma by using \cite[Proposition 2.4.20, Proposition 2.4.21]{KL1}.	
\end{proof}

\indent The corresponding glueing finitely generated modules is also achieved in current noetherian situation. First we consider the corresponding result:

\begin{proposition}
The descent along the morphism $W\rightarrow B_1\bigoplus B_2$ is effective as long as one focuses on the category of all the finitely presented modules over $W\widehat{\otimes}V$. 	
\end{proposition}

\begin{proof}
This is the corresponding noetherian implication of the corresponding result in \cite[Lemma 2.14]{TX2}.	
\end{proof}

\begin{proposition}
The corresponding functor of taking the corresponding global section will give rise to the corresponding equivalence between the corresponding $V$-coherent sheaves and the $V$ finitely presented modules.
\end{proposition}

\begin{proof}
This is also the corresponding noetherian implication of the corresponding result in \cite[Theorem 2.15]{TX2}.	
\end{proof}

\indent Then fixing a stable basis $\mathbb{H}$ for the \'etale site of the space $\mathrm{Spa}(W,W^+)$, then we have the results over the \'etale site as well:

\begin{proposition}
We consider the corresponding presheafification of any finitely generated module $M$ over $W\widehat{\otimes}V$, to be more precise over $\mathrm{Spa}(W,W^+)$ we will define the corresponding presheaf $\widetilde{M}$ by taking the inverse limit of the base changes of $M$ throughout all the member in the basis $\mathbb{H}$. Then we have that the Tate glueing property holds in our current situation for such presheaf $\widetilde{M}$.
\end{proposition}

\begin{proposition}
The descent (in the \'etale topology) along the morphism $W\rightarrow B_1\bigoplus B_2$ is effective as long as one focuses on the category of all the finitely presented modules over $W\widehat{\otimes}V$. 	
\end{proposition}

\begin{proof}
This is the corresponding noetherian implication of the corresponding result in \cite[Lemma 2.22]{TX2}.	
\end{proof}

\begin{proposition} \label{proposition6.9}
The corresponding functor (in the \'etale topology) of taking the corresponding global section will give rise to the corresponding equivalence between the corresponding $V$-coherent sheaves and the $V$ finitely presented modules.
\end{proposition}

\begin{proof}
This is also the corresponding noetherian implication of the corresponding result in \cite[Theorem 2.23]{TX2}.	
\end{proof}

\

\subsection{Application to Descent over Noncommutative $\infty$-Analytic Prestacks after Bambozzi-Ben-Bassat-Kremnizer}

\indent We now contact with the corresponding derived analytic spaces from \cite{BBBK}. Recall that we have the categories $\mathrm{Simp}(\mathrm{Ind}^m(\mathrm{BanSets}_{E}))$ and $\mathrm{Simp}(\mathrm{Ind}(\mathrm{BanSets}_{E}))$ of the corresponding simplicial sets in the corresponding inductive category of Banach sets over $E=\mathbb{Q}_p,\mathbb{F}_p((t))$.

\begin{theorem}\mbox{\bf{(Bambozzi-Ben-Bassat-Kremnizer)}}
The categories $\mathrm{Simp}(\mathrm{Ind}^m(\mathrm{BanSets}_{E}))$ and $\mathrm{Simp}(\mathrm{Ind}(\mathrm{BanSets}_{E}))$ admit symmetric monoidal model categorical structure.
	
\end{theorem}

\indent Therefore based on this nice structure \cite{BBBK} defined the corresponding $\infty$-categories of $\mathbb{E}_\infty$-rings:
\begin{align}
\mathrm{sComm}(\mathrm{Simp}(\mathrm{Ind}^m(\mathrm{BanSets}_{E})))\\
\mathrm{sComm}(\mathrm{Simp}(\mathrm{Ind}(\mathrm{BanSets}_{E}))).	
\end{align}

\indent We now consider the corresponding some object $A$ which is a corresponding $\infty$-locally convex ring in the categories above and we consider the corresponding object in the corresponding opposite category we call that $\mathrm{Spec}A$. And we consider the corresponding homotopy Zariski topology, which allows us to talk about the corresponding $\infty$-analytic stacks.

\indent Now recall that a connective $\mathbb{E}_1$-ring is called noetherian if we have that $\pi_0(A)$ is noetherian and $\pi_n(A)$ is basically finitely generated over $\pi_0(A)$.

\begin{remark}
The corresponding Koszul simplicial Banach ring considered in \cite{BK} is actually concentrated in the nonpositive degrees\footnote{Thanks for Federico Bambozzi for reminding us of this.}, in the cohomological convention, namely $\mathrm{H}^{-n}(.)=0$ for $n<0$. But as in any abstract homotopy theory we consider the corresponding conventional transformation:
\begin{align}
\pi_n:=	\mathrm{H}^{-n},n\geq 0.
\end{align} 	
\end{remark}

\indent Now we consider an $\infty$-analytic stack $X$ as considered in \cite{BBBK}. And we consider the corresponding ringed site attached to $X$ under the corresponding homotopy Zariski topology, denoted by $(X,\mathcal{O}_X)$.

\begin{example}
The examples which are very interesting to us are the corresponding $\infty$-Huber spectra constructed from any Banach rings over $E$ from \cite{BK} by using Koszul complex. The corresponding classical sheafiness issue could be really forgotten.	
\end{example}

\indent Now we consider the corresponding $\mathbb{E}_1$-ring $A$ in the corresponding categories $\mathrm{Simp}(\mathrm{Ind}^m(\mathrm{BanSets}_{E}))$ and $\mathrm{Simp}(\mathrm{Ind}(\mathrm{BanSets}_{E}))$. We make the following assumption:

\begin{assumption}
We assume that all the ring objects below are noetherian, as $\mathbb{E}_1$-rings.  	
\end{assumption}

\indent Suppose now we have four $\mathbb{E}_1$ Banach rings $A,A_1,A_2,A_{12}$\footnote{Namely all the homotopy groups have to be Banach instead of Bornological or ind-Banach.} in $\mathrm{Simp}(\mathrm{Ind}^m(\mathrm{BanSets}_{E}))$ and $\mathrm{Simp}(\mathrm{Ind}(\mathrm{BanSets}_{E}))$ respectively. And we assume that we have the following short strictly exact sequence in the sense of a glueing square:
\[
\xymatrix@C+0pc@R+0pc{
0 \ar[r] \ar[r] \ar[r] &\pi_0(A) \ar[r] \ar[r] \ar[r] &\pi_0(A_1)\oplus \pi_0(A_2)  \ar[r] \ar[r] \ar[r] &\pi_0(A_{12}) \ar[r] \ar[r] \ar[r] &0 ,   
}
\]
and we assume that the image of $\pi_0(A_i)$ in $\pi_0(A_{12}) $ is dense for each $i=1,2$. And we assume the corresponding morphism $\pi_0(A_i)\rightarrow \pi_0(A_{12})$ is flat for $i=1,2$. And we assume that $A,A_1,A_2,A_{12}$ form the corresponding derived glueing sequence.

\begin{example}
One can construct the following example. First consider any glueing sequence of Banach rings (not required to be sheafy or commutative) over $\mathbb{Q}_p$:
\[
\xymatrix@C+0pc@R+0pc{
0 \ar[r] \ar[r] \ar[r] &\Pi \ar[r] \ar[r] \ar[r] &\Pi_1\oplus \Pi_2  \ar[r] \ar[r] \ar[r] &\Pi_{12} \ar[r] \ar[r] \ar[r] &0.  
}
\]	
This is very natural in the corresponding situation for example where we deform from a corresponding nice short exact sequence of commutative sheafy rings of the same type, but in any rate we do not require the corresponding commutativity or the correponding sheafyness. Then take any derived global section of some $\mathrm{Spa}^h(R)$ from \cite{BK}, denoted by $R^h$. Over $\mathbb{Q}_p$, we have the following short strictly exact sequence:
\[
\xymatrix@C+0pc@R+0pc{
0 \ar[r] \ar[r] \ar[r] &\Pi\widehat{\otimes}_{\mathbb{Q}_p} \pi_0(R^h) \ar[r] \ar[r] \ar[r] &\Pi_1\widehat{\otimes}_{\mathbb{Q}_p} \pi_0(R^h)\oplus \Pi_2\widehat{\otimes}_{\mathbb{Q}_p} \pi_0(R^h)  \ar[r] \ar[r] \ar[r] &\Pi_{12}\widehat{\otimes}_{\mathbb{Q}_p} \pi_0(R^h) \ar[r] \ar[r] \ar[r] &0 ,   
}
\]
but since the corresponding ring $\Pi,\Pi_1,\Pi_2,\Pi_{12}$ will actually realize the situation where $\Pi_*\widehat{\otimes} \pi_0(R^h)\overset{\sim}{\rightarrow}\pi_0(\Pi_*\widehat{\otimes}^\mathbb{L}R^h)$, we will have the desired situation as long as restrict to now the noetherian situation.
\end{example}

\begin{conjecture}
The map $A\rightarrow \prod_{i=1,2} A_i$ is an effective descent morphism with respect to the finitely presented left module spectra (meaning we have finitely presented $\pi_0$).	
\end{conjecture}

\begin{proposition}
The operation of taking equalizer along the corresponding map $\prod_{i=1,2} A_i\rightarrow A_{12}$ preserves the property of being finitely presented.	
\end{proposition}

\begin{proof}
The situation where we deform from the corresponding $\mathbb{E}_\infty$ by some noncommutative deformation could be implied by \cite[Lemma 2.14]{TX2}. While in general this follows from \cite[Lemma 6.83]{TX4}. Note that we achieve the finiteness for each $\pi_n(M)$ just by application of the results we know in the classical deformed situation, which is because the connecting homomorphism vanishes on each level.
\end{proof}

\indent What could happen in the commutative setting is actually also interesting. We have the following example in mind:

\begin{example}
Consider the context of the previous example, but assume that all the rings involved are Banach commutative. First we consider over $\mathbb{Q}_p$:
\[
\xymatrix@C+0pc@R+0pc{
0 \ar[r] \ar[r] \ar[r] &\Pi \ar[r] \ar[r] \ar[r] &\Pi_1\oplus \Pi_2  \ar[r] \ar[r] \ar[r] &\Pi_{12} \ar[r] \ar[r] \ar[r] &0,  
}
\]		
a strictly exact sequence of sheafy rings. We now deform this along some spectrum $\mathrm{Spa}^h(R)$ in \cite{BK} which will produce a very interesting situation where we can actually obtain desired situation for the descent of finitely presented module spectra. And note that in commutative setting we could also have more geometric contexts to rely on as in \cite{BK} and \cite{BBBK}. This may have contact with the corresponding derived Galois deformation theory of \cite{GV} for instance by considering the corresponding simplicial pro-Artin rings.
\end{example}

\begin{remark}
The corresponding machinery from \cite{CS} should definitely reflect similar things here even in the $\mathbb{E}_1$-ring context. Also we would like to mention in the commutative setting that actually the descent for certain quasi-coherent modules are also considered extensively in \cite{BBK}. Strikingly the ideas in \cite{BBK} (although developed in a quite commutative setting) come from partially work from Kontsevich-Rosenberg \cite{KR} namely essentially the noncommutative descent. Our feeling is that definitely the descent for certain quasi-coherent modules in \cite{BBBK} could be established to noncommutative setting in some form both in the archimedean and nonarchimedean situations. 
\end{remark}




\newpage

\subsection*{Acknowledgements} 
We would like to thank Professor Kedlaya for helpful discussion on the corresponding materials in the chapter \cite{Ked1}, especially the discussion beyond the book which helps us have the chance to enrich the presentation here.


\bibliographystyle{ams}

\end{document}